\documentclass[a4paper,11pt]{article}

\usepackage[latin1]{inputenc}
\usepackage[T1]{fontenc}
\usepackage[english]{babel}

\usepackage{mathrsfs}
\usepackage{amscd}
\usepackage{amsfonts}
\usepackage{amsmath}
\usepackage{amssymb}
\usepackage{amstext}
\usepackage{amsthm}
\usepackage{amsbsy}

\usepackage{xspace}
\usepackage[all]{xy}
\usepackage{graphicx}
\usepackage{url}
\usepackage{latexsym}

\usepackage{graphicx} % support the \includegraphics command and options

\usepackage{booktabs} % for much better looking tables
\usepackage{array} % for better arrays (eg matrices) in maths
\usepackage{paralist} % very flexible & customisable lists (eg. enumerate/itemize, etc.)
\usepackage{verbatim} % adds environment for commenting out blocks of text & for better verbatim
\usepackage{subfig} % make it possible to include more than one captioned figure/table in a single float
\usepackage{fancyhdr} % This should be set AFTER setting up the page geometry
\usepackage{sectsty}
\allsectionsfont{\sffamily\mdseries\upshape} % (See the fntguide.pdf for font help)
\usepackage[nottoc,notlof,notlot]{tocbibind} % Put the bibliography in the ToC
\usepackage[titles,subfigure]{tocloft} % Alter the style of the Table of Contents

\usepackage[textwidth=100pt,textsize=footnotesize,bordercolor=white,color=blue!30]{todonotes}
\usepackage{hyperref} 

\makeatletter
\newcommand*{\rom}[1]{\expandafter\@slowromancap\romannumeral #1@}
\makeatother

\theoremstyle{definition}

\newtheorem{fact}{fact}

\newtheorem{thm}[fact]{Theorem}
\newtheorem{lemma}[fact]{Lemma}

\newtheorem{corollary}[fact]{Corollary}
\newtheorem{defini}[fact]{Definition}

\title{The Lost Melody Theorem for Infinite Time Blum-Shub-Smale Machines}
\author{Merlin Carl}
\date{}

\begin{document}

\maketitle
%? - duck

\begin{abstract}
	We consider recognizability for Infinite Time Blum-Shub-Smale machines, a model of infinitary computability introduced in Koepke and Seyfferth \cite{KS}. In particular, we show that the lost melody theorem (originally proved for ITTMs in Hamkins and Lewis \cite{HL}), i.e. the existence of non-computable, but recognizable real numbers, holds for ITBMs, that ITBM-recognizable real numbers are hyperarithmetic and that both ITBM-recognizable and ITBM-unrecognizable real numbers appear at every level of the constructible hierarchy below $L_{\omega_{1}^{\text{CK}}}$ at which new real numbers appear at all.
\end{abstract}

\section{Introduction}

In ordinal computability, a considerable variety of machine models of infinitary computability was defined, including Infinite Time Turing Machines (ITTMs), (weak) Infinite Time Register Machines (ITRMs), Ordinal Turing Machines (OTMs) and Ordinal Register Machines (ORMs) etc. For each of these models, a real number (more generally, a set of ordinals) $x$ is called ``recognizable'' if and only if there is a program $P$ such that, when executed on a machine of the type under consideration, $P$ halts on every input $y$ with output $0$ or $1$ and outputs $1$ if and only if $x=y$. The term was originally defined for ITTMs in Hamkins and Lewis \cite{HL}, where the most prominent statement about ITTM-recognizability was proved, namely the existence of real numbers that are ITTM-recognizable, but not ITTM-computable, so called ``lost melodies''. 

Later on, recognizability was also studied for other machine models. The lost melody theorem was shown to also hold for ITRMs (see \cite{CFKMNW}; see \cite{Ca1} and \cite{Ca2} for a detailed study of ITRM-recognizability) and OTMs with parameters (where computability amounts to constructibility, while recognizability takes us up to $M_{1}$, the canonical inner model for a Woodin cardinal, see \cite{CSW}). On the other hand, it fails for OTMs without parameters and weak ITRMs, see \cite{Ca3}.
%Lost Melodies for ITTMs, ITRMs, pOTMs. None for OTMs, weak ITRMs. The point with "domain enumeration".

%Infinite Time Blum-Shub-Smale-machines. What is different for these.

Infinite Time Blum-Shub-Smale machines, introduced in Koepke and Seyfferth \cite{KS} are register models of infinitary computability that compute with real numbers rather than ordinals as their register contents. 
ITBMs are known to compute exactly the real numbers in $L_{\omega^{\omega}}$ by Welch \cite{We} and Koepke and Morozov \cite{KM}. Moreover, it is known from Koepke and Seyfferth \cite{KS} (Theorem $1$) that an ITBM-program with $n$ nodes either halts in $<\omega^{n+1}$ many steps or not at all. So far, recognizability for ITBMs was not considered. Indeed, as ITBMs are extremely weak in comparison with the other models mentioned above, many of the usual methods for studying recognizability are not available in this setting.

In this paper, we close this gap by (i) showing that the lost melody theorem holds for ITBMs and in particular the ITBM-recognizability of the ITBM-halting number, (ii) showing $L_{\omega_{1}^{\text{CK}}}$ to be the minimal $L$-level containing all ITBM-recognizable real numbers and (iii) that both new ITBM-recognizable and new ITBM-unrecognizable real numbers appear at every index level after $\omega^{\omega}$ below $\omega_{1}^{\text{CK}}$. 

Most arguments in this paper are variants of the corresponding arguments used in the investigation of register models of ordinal computability, specifically Infinite Time Register Machines (ITRMs, see Koepke and Miller \cite{KM}) and weak Infinite Time Register Machines (now called wITRMs, see Koepke \cite{Ko1}). However, due to the weakness of ITBMs, considerable adaptations are required. In this respect, ITBMs turn out to be a kind of mixture between these two machine types with respect to recognizability: Like ITTMs and ITRMs but other than wITRMs, they have lost melodies, even though they are too weak to check whether a given real number codes a well-ordering (which is crucial in the constructions for ITRMs and ITTMs). The real number coding the ITBM-halting problem is ITBM-recognizable, which is also true for ITRMs, but fails for ITTMs. The distribution of the ITBM-recognizable real numbers in G\"odel's constructible hierarchy $L$ is different for ITBMs than for all other machine types considered so far: From $\omega^{\omega}$ up to $\omega_{1}^{\text{CK}}$, new unrecognizable and new recognizable real numbers occur at every level at which new real numbers occur at all, while for ITTMs and ITRMs, there are ``gaps'' in the set of levels at which new recognizable real numbers are constructed. 
%[They have lost melodies, but no well-foundedness check (first time this happens). All recognizables are hyperarithmetical. Halting real is recognizable.]

An ordinal $\alpha$ is called an ``index'' if and only if $L_{\alpha+1}\setminus L_{\alpha}$ contains a real number. By standard fine-structure (see, e.g., Jensen \cite{Je}), $L_{\alpha+1}$ contains a bijection $f:\omega\rightarrow L_{\alpha}$ when $\alpha$ is an index. Moreover, by Theorem $1$ of Boolos and Putnam \cite{BP}, if $\alpha$ is an index, then $L_{\alpha+1}$ contains an ``arithmetical copy'' of $L_{\alpha}$, i.e. a real number coding $L_{\alpha}$. Below, unless indicated otherwise, $p$ will denote Cantor's pairing function. %Moreover, we fix a natural bijection $s$ from $\omega$ to the set of finite sequences of natural numbers that do not start with $0$.
	
\subsection{Infinite Time Blum-Shub-Smale Machines}

Infinite Time Blum-Shub-Smale machines were introduced in Koepke and Seyfferth (\cite{KS}) and then studied further in Koepke and Morozov \cite{KM} and Welch \cite{We}. We briefly recall the definitions and results required for this article.

Like a Blum-Shub-Smale machine, an ITBM has finitely many registers, each of which can store a single real number. An ITBM-program is just an ordinary Blum-Shub-Smale-machine program, i.e. a finite, numerated list of commands for applying a rational functions to the contents of some registers and (i) replacing the content of some register with the result or (ii) jumping to some other program line, depending on whether the value of the function is positive or not; this latter kind of command is called a ``node''. At successor times, an ITBM works like a BSSM, while at limit levels, the active program line is the inferior limit of the sequence of earlier program lines and the content of each register $R$ is the Cauchy limit of the sequence of earlier contents of $R$, provided this sequence converges; if this sequence does not converge for some register, the computation is undefined.

We fix a natural enumeration $(P_{i}:i\in\omega)$ of the ITBM-programs. For an ITBM-program $P$ and a real number $x$, we write $P^{x}$ for the computation of $P$ that starts with $x$ in the first register.

\begin{defini}
A real number $x$ is ITBM-computable if and only if there is an ITBM-program $P$ that starts with $0$s in all of its registers and halts with $x$ in its first register.

We say that a real number $x$ is ITBM-recognizable if and only if there is an ITBM-program $P$ such that, for all real numbers $y$, $P^{y}$ halts with output $1$ if and only if $y=x$ and otherwise, $P^{y}$ halts with output $0$.
\end{defini}

We summarize the relevant results about ITBMs in the following theorem.

\begin{thm}{\label{basic facts}}
	
	(i) (Koepke, Seyfferth, \cite{KS}) If $P$ is an ITBM-program using $n\in\omega$ many nodes and $x$ is a real number, then $P^{x}$ halts in $<\omega^{n+1}$ %sollte es n sein oder (n+1)? 
	many steps or it does not halt at all. In particular, any ITBM-program $P^{x}$ either halts in $<\omega^{\omega}$ many steps or not at all. An ordinal $\alpha$ is ITBM-clockable if and only if $\alpha<\omega^{\omega}$.
	
	(ii) (Koepke, Morozov \cite{KM}, Welch \cite{We}) A real number $x$ is ITBM-computable from the real input $y$ if and only if $x\in L_{\omega^{\omega}}[x]$. In particular, $x$ is ITBM-computable if and only if $x\in L_{\omega^{\omega}}$.
\end{thm}

%Beachte: Wie bei ITRMs sind Berechnungen stets ``sicher'', d.h. man kann es dahin bringen, dass sie auf jedem Input halten, indem man einen Halteprobleml\"oser f\"ur die dem fraglichen Programm entsprechende Zahl von Verzweigungspunkten im Programm (``nodes'') vorausschickt.

As a consequence of (i), it is possible to decide, for every ITBM-program $P$, the set $\{x\subseteq\omega:P^{x}\text{ halts}\}$ on an ITBM: Namely, if $P$ uses $n$ nodes, simply run $P^{x}$ for $\omega^{n+1}$ many steps and see whether it has halted up to this point. Thus, if a partial function $f:\mathbb{R}\rightarrow\mathbb{R}$ is ITBM-computable, there is also a total ITBM-computable function $\hat{f}:\mathbb{R}\rightarrow\mathbb{R}$ such that $\hat{f}(x)=f(x)$ whenever $f(x)$ is defined and otherwise $\hat{f}(x)=0$. These properties of ITBMs will be freely used below.

\section{The Lost Melody Theorem for ITBMs}

In this section, we will show that there is a real number $x$ which is ITBM-recognizable, but not ITBM-computable. 

Let $x$ be a real number with the following properties:

\begin{enumerate}
	\item There is a bijection $f:\omega\rightarrow L_{\omega^{\omega}}$ such that $x=\{p(i,j):f(i)\in f(j)\}$. We fix $f$ from now on.
	\item $x\in L_{\omega^{\omega}}+1$. In particular, $x$ is definable over $L_{\omega^{\omega}}$, and in fact definable without parameters (by fine-structure). Let $\phi$ be an $\in$-formula such that $x=\{i\in L_{\omega^{\omega}}:L_{\omega^{\omega}}\models\phi(i)\}$.
	\item $f[\{2i:i\in\omega\}]=\omega^{\omega}$; that is, ordinals are coded exactly by the even numbers.
	\item The real number $c:=\{p(i,j):p(2i,2j)\in x\}$ (which, by definition, is a code of $\omega^{\omega}$) is recursive. 
\end{enumerate}

\begin{lemma}{\label{early decoding}}
Let $c\subseteq\omega$ be such that, for some ordinal $\alpha$ and some bijection $f:\omega\rightarrow\alpha$, we have $c=\{p(i,j):i,j\in\omega\wedge f(i)\in f(j)\}$. Then $f\in L_{\alpha+1}[c]$. 
In particular, if $c$ is recursive and $\alpha>\omega+2$, then $f\in L_{\alpha+1}$.
%wenn c rekursiv ist, wird $L_{\alpha+1}[c]=L_{\alpha+1} sein.
\end{lemma}
\begin{proof}
We need to show that $f$ is definable over $L_{\alpha}[c]$. First suppose that $\alpha$ is a limit ordinal. 
Then $f$ is defined as follows. For $i\in\omega$, we have $f(i)=\beta$ if and only if there is a sequence $(a_{\iota}:\iota\leq\beta)$ of natural numbers with the following properties: 

\begin{enumerate}
\item For all $i\in\{a_{\iota}:\iota\leq\beta\}=:A$, and all $j\in\omega$, if $p(j,i)\in c$, then $j\in A$
	
\item for all $\iota,\xi\leq\beta$, we have $\iota<\xi$ if and only if $p(a_{\iota},a_{\xi})\in c$

\item $a_{\beta}=i$
\end{enumerate}

When $\alpha$ is a limit ordinal, these sequences will be contained in $L_{\alpha}$, so the above provides a definition of $f$ over $L_{\alpha}$. When $\alpha$ is a successor ordinal, the above works up to the last limit ordinal before $\alpha$ and then the remaining values of $f$ can be defined separately explicitly; we skip the details of this case.
	
%dieser beweis sollte wohl in den anhang, wenn es ein CiE-artikel wird.
The second claim now follows from the first as a recursive real number $c$ is contained in $L_{\omega+1}$, so that $L_{\alpha+1}[c]=L_{\alpha+1}$ when $\alpha>\omega+2$.
\end{proof}

\begin{lemma}{\label{code existence}}
	There is a real number $x$ satisfying (1)-(4) above.
\end{lemma}
\begin{proof}
	It is clear that the Skolem hull of the empty set in $L_{\omega^{\omega}}$ is equal to $L_{\omega^{\omega}}$. By standard fine-structure (see \cite{Je}), this implies that $L_{\omega^{\omega}+1}$ contains a bijection $g:\omega\rightarrow L_{\omega^{\omega}}$. 

	%Fix a bijection $h:\omega\rightarrow\omega^{\omega}$, given by $h(s(a_{1},...,a_{k}))=\omega^{k}\cdot a_{1}+\omega^{k}\cdot a_{1}+...+\omega^{0}\cdot a_{k}$.
	Moreover, as $\omega^{\omega}<\omega_{1}^{\text{CK}}$, there is a recursive code $c$ for $\omega^{\omega}$. Using Lemma \ref{early decoding}, a function $h:\omega\rightarrow\omega^{\omega}$ such that $c=\{p(i,j):h(i)\in h(j)\}$ is definable over $L_{\omega^{\omega}}$. 
	
	%problem: woher wissen wir, dass der rekursive code c aus einer �ber der relevanten L-stufe (hier: L_{\omega^{\omega}}) definierbaren bijektion entstanden ist?
	%ist die decodierung eines codes f�r \alpha �ber L_{\alpha} stets definierbar?
	%ansatz: d(i)=\beta falls eine folge nat�rlicher zahlen der l�nge \beta existiert, in dieser das i
	%an letzter stelle steht, die nat�rlichen zahlen im sinne von c ein anfangsst�ck bilden und "richtig" geordnet sind.
	%sind diese folgen immer da? sicher, wenn es eine limesstufe ist. ansonsten evtl. die verbleibenden nachfolgerstufen von hand "draufmontieren". das geht ja im rahmen einer definition.
	
	%CHECK THE DETAILS!!! IST DAS JETZT SO, WIE ES SEIN SOLLTE?
	Now define $f:\omega\rightarrow L_{\omega^{\omega}}$ by letting, for $i\in\omega$, $f(2i)=h(i)$ 
	%be the $h$-image of the $i$th natural number whose $g$-image is an ordinal 
	and letting $f(2i+1)$ be the $g$-image of the $i$th natural number whose $g$-image is not an ordinal. Since $g$ is definable over $L_{\omega^{\omega}}$, so is $f$. 
%[das reicht nicht, du musst das bild der geraden zahlen noch rekursiv, genauer gleich $c$ kriegen! dazu muss man noch einmal umpermutieren.] 
Now let $x:=\{p(i,j):i,j\in\omega\wedge f(i)\in f(j)\}$. Then $x$ is definable over $L_{\omega^{\omega}}$ and by definition as desired.
	%weiterer ansatz: schicke 2i auf die i-te natuerliche zahl, die von g auf eine ordinalzahl abgebildet wird; schicke (2i+1) auf die i-te natuerliche zahl, die von g nicht auf eine ordinalzahl abgebildet wird. 
    %da g definierbar ist, ist $f$ es auch.
    %da x aus f definierbar ist, ist x es auch.
\end{proof}

We now show that $x$ is a lost melody for ITBMs.

\begin{lemma} (Truth predicate evaluation)
	Given a real number $y$ coding the structure $(Y,R)$ (with $Y$ a set, $R\subseteq Y\times Y$ a binary relation on $Y$, $g:\omega\rightarrow Y$ a bijection and $y=\{p(i,j):(f(i),f(j))\in Y\}$)
	there is an ITBM-program $P_{\text{truth}}$ that compute the truth predicate over $(Y,R)$ (i.e., for each $\in$-formula $\phi$ and all $i_{1},...,i_{n}\in\omega$, $P_{\text{truth}}$ will decide, on inputs $y$ and $(\phi,(i_{1},...,i_{n}))$, whether or not $(Y,R)\models\phi(g(i_{1}),...,g(i_{n}))$.
	% and $n\in\omega$, the $\Sigma_{n}$-truth predicate for propositions in the language of set theory for $(Y,R)$ is ITBM-computable in $y$.
\end{lemma}
\begin{proof}
	%By exhaustive search and induction. DO IT!!!
	%``morally'' contained in Koepke-Morozov.
	By Proposition 2.7 of Koepke and Morozov \cite{KM}, there is an ITBM-program $P$ such that, for each input $y\subseteq\omega$, $P^{y}$ computes the (classical) Turing-jump of $y$. By the iteration lemma in \cite{KM}, there is also an ITBM-program $H$ that computes the $\omega$-th iteration $y^{(\omega)}$ of the Turing-jump of $y$. But it is clear that the truth predicate for $(Y,R)$ is recursive in $y^{(\omega)}$, and a fortiori ITBM-computable. 
	
	%koepke-morozov, prop. 2.7: zu jeder reellen zahl kann man den Turing-Jump berechnen. Also auch jede endliche Iteration. Das sollte reichen, um beschr\"ankte Wahrheitspr\"adikate zu berechnen.
\end{proof}

\begin{corollary} (Identification of natural numbers)
	There is a program that identifies the natural numbers coding natural numbers in a real code $r$ for a structure $(A,R)$. Moreover, there is a program $P_{\text{id}}$ that, for each natural number $k$, identifies the natural number $i$ that codes $k$ in the sense of $r$, provided such $i$ exists.
\end{corollary}
\begin{proof}
	The first part is an immediate consequence of the last lemma. 
	
	For the second part, note that there is a recursive function that maps each $k\in\omega$ to an $\in$-formula $\psi_{k}$ such that $\psi_{k}(x)$ holds if and only if $x=k$.\footnote{For example, we can take $\psi_{0}(x)$ to be $x\neq x$ and then let $\psi_{k+1}(x)$ be $\forall{y}(y\in x\leftrightarrow(\exists{z}(\psi_{k}(z)\wedge y\in z)\vee\psi_{k}(y)))$.} But then, searching for the natural number coding $k$ is an easy application of the last lemma: Just check successively, for each $i\in\omega$, whether $\psi_{k}(g(i))$ holds in $(A,R)$ and output the first $i\in\omega$ for which it is true.
	
	%MORE HERE: endliche suche ist unproblematisch, man testet die 
	
	%Find and store $0$, $1$ etc. successively, using the ``halving trick'' above. Then just look it up there. The second part is even easier.
	
	%Sollte auch direkt mit dem Turing-jump gehen.
	%man kriegt auch den $\omega$-ten jump. damit kriegt man wohl direkt ein vollst\"andiges wahrheitspr\"adikat?
\end{proof}

\begin{lemma}
	$x$ is not ITBM-computable.
\end{lemma}
\begin{proof}
	Since ITBM-halting times are bounded by $\omega^{\omega}$, $L_{\omega^{\omega}}$ contains all halting ITBM-computations. Thus, the statement that the $i$th ITBM-program $P_{i}$ halts is $\Sigma_{1}$ over $L_{\omega^{\omega}}$. By bounded truth predicate evaluation, the set $H$ of $i\in\omega$ for which $P_{i}$ halts - i.e., the ITBM halting set - is thus ITBM-computable from $x$. By Koepke-Morozov (transitivity lemma), $H$ would thus be ITBM-computable if $x$ was ITBM-computable. Thus, $x$ is not ITBM-computable.
\end{proof}

\begin{lemma}{\label{x recog}}
	$x$ is ITBM-recognizable.
\end{lemma}
\begin{proof}
	Let the input $y$ be given.
	First, use truth predicate evaluation to check whether $y$ codes a model of $V=L$. If not, output $0$.
	
	If yes, check whether, in $y$, $i\in\omega$ codes an ordinal if and only if $i$ is even. This can be determined by computing $y^{(\omega)}$ in which the set $s$ of natural numbers coding ordinals in the sense of $y$ is recursive, and then checking whether the Turing program that (in the oracle $s$) runs through $\omega$ and halts once it has found an odd number in $s$ or an even number not in $s$ halts.
%Again, this works by truth predicate evaluation for the formula ``$z$ is an ordinal'', applied to each natural number in turn as in the proof of bounded truth predicate evaluation (shrink everything to an interval of the form $(\frac{1}{2^n},\frac{1}{2^{n+1}})$, shrink by $\frac{1}{2}$ more each time). %[MORE HERE!!!] 
	If not, return $0$. Otherwise, continue. 
	
	Check whether $\{p(i,j):p(2i,2j)\in y\}=c$. This is possible as $c$ is recursive, so we can just compare. If not, return $0$. Otherwise, we know that $y$ codes $L_{\omega^{\omega}}$, and we only need to check whether it is the ``right'' code. To do this, we continue as follows:
	
	Using bounded truth predicate evaluation and identification of natural numbers, compute the set $s$ of natural numbers $i$ such that $L_{\omega^{\omega}}\models\phi(i)$. At this point, we konw that $s=x$. Now simply compare $s$ to $y$. If they are equal, output $1$, otherwise output $0$.
\end{proof}

We note for later use that the proof of lemma \ref{x recog} shows more: 

\begin{corollary}{\label{many recognizable codes}}
	Let $\alpha<\omega_{1}^{\text{CK}}$ be an index. Then $L_{\alpha}$ has an ITBM-recognizable code $c$. In fact, $c$ can be taken to be contained in $L_{\alpha+1}$.
\end{corollary}
\begin{proof}
	Since $\alpha<\omega_{1}^{\text{CK}}$, there is a recursive real number $r$ that codes $\alpha$. But then, there is, as in Lemma \ref{code existence}, a code $c$ for $L_{\alpha}$ that is (i) contained in $L_{\alpha+1}$, (ii) codes ordinals by even numbers and such that (iii) $\{p(i,j):i,j\in\omega\wedge p(2i,2j)\in c\}=r$. Since $L_{\alpha+1}$ contains a bijection $f:\omega\rightarrow L_{\alpha}$, it is easy to see that we can take $c$ to be an element of $L_{\alpha+1}$. Now $c$ is recognizable as in the proof of Lemma \ref{x recog}. 
\end{proof}

Thus, we obtain:

\begin{thm}
	There is a lost melody for Infinite Time Blum-Shub-Smale machines.
\end{thm}

%\subsection{The ITBM-Halting Real is Recognizable}

It is known from \cite{Ca2} that the set of indices of halting ITRM-programs is ITRM-recognizable, while the set of indices of halting ITTM-programs is not ITTM-recognizable. Here, we show that ITBMs resemble ITRMs in this respect: Namely, define $H$ to be the set of natural numbers $i$ such that $P_{i}$ halts. It is not hard to see, (though a bit cumbersome) that a code $c$ for $L_{\omega^{\omega}}$ is ITBM-computable from $H$, say by the program $P$. Now, to identify whether a given real number $x$ is equal to $H$, first check, using the bounded halting problem solver, whether $P^{x}$ will halt. If not, output $0$ and halt. If yes, let $y$ be the output of $P^{x}$ and check, as in the proof of Lemma \ref{x recog}, whether $y$ is a code for $L_{\omega^{\omega}}$. If not, output $0$ and halt. Otherwise, use $y$ to compute, again as in the proof of Lemma \ref{x recog}, the set $H$, (which is $\Sigma_{1}$ over $L_{\omega^{\omega}}$) and compare $x$ to $H$.
Thus, we get:

\begin{thm}
	The real number $H$ coding the halting problem for ITBMs is ITBM-recognizable.
\end{thm}

%\begin{thm}
%Define $H$ to be the set of natural numbers $i$ such that $P_{i}$ halts. Then $H$ is ITBM-recognizable. 
%end{thm}
%\begin{proof}
	
%It suffices to show that a code $c$ for $L_{\omega^{\omega}}$ is ITBM-computable from $H$. For let $P$ be an ITBM-program such that $P^{H}$ computes $c$. To identify whether a given real number $x$ is equal to $H$, first check, using the bounded halting problem solver, whether $P^{x}$ will halt. If not, output $0$ and halt. If yes, let $y$ be the output of $P^{x}$ and check, as in the proof of Lemma \ref{x recog}, whether $y$ is a code for $L_{\omega^{\omega}}$. If not, output $0$ and halt. Otherwise, use $y$ to compute, again as in the proof of Lemma \ref{x recog}, the set $H$, (which is $\Sigma_{1}$ over $L_{\omega^{\omega}}$) and compare $x$ to $H$.

%It remains to show that a code for $L_{\omega^{\omega}}$ is computable from $H$. %DO IT!!!

%%Question: Recognizability of the halting real $H$? Distribution? 
%%hierher

%Vermutung: $H$ ist recognizable, weil man aus $H$ recht leicht (auslesen und zusammenf\"ugen) einen Code f\"ur $L_{\omega^{\omega}}$ berechnen, diesen identifizieren kann (wie oben beschrieben - er muss nicht eindeutig identifiziert werden, es reicht, zu wissen, dass es einer ist) und dann dar\"uber die $\Sigma_{1}$-Formel auswerten kann, die $H$ \"uber $L_{\omega^{\omega}}$ definiert. 
%\end{proof}

\section{The Distribution of ITBM-Recognizable Real Numbers}

Where do ITBM-recognizable real numbers occur in $L$? This question was studied in detail in \cite{Ca1} and \cite{Ca2} for the case of ITRMs and in \cite{Ca3} for wITRMs, where it turned out that the wITRM-recognizable real numbers coincide with the wITRM-computable real numbers (i.e., there are no lost melodies for wITRMs), which is known from Koepke \cite{Ko1} to coincide with the hyperarithmetical real numbers. By an adaptation of the proof in \cite{Ca3}, we obtain:

\begin{lemma}{\label{upper bound}}
	Let $x\subseteq\omega$ be ITBM-recognizable. Then $x\in L_{\omega_{1}^{\text{CK}}}$. 
\end{lemma}
\begin{proof}
Let $x\subseteq\omega$ be ITBM-recognizable, and let $P$ be an ITBM-program that recognizes $x$. It follows that there is an ordinal $\gamma<\omega^{\omega}$ such that $P^{x}$ halts in exactly $\gamma$ many steps. As a snapshot of an ITBM can easily be encoded as a real number, the same holds for a $\gamma$-sequence of such snapshots.

Now, the statement ``There is a real number $g$ such that $g$ codes an ITBM-computation of length $\gamma$ in the oracle $y$ that halts with output $1$'' is $\Sigma_{1}^{1}$; thus, the set of such $y$ is $\Sigma_{1}^{1}$, and, in particular, $\{x\}$ is $\Sigma_{1}^{1}$. %[VERGLEICHE MIT KOEPKE-MOROZOV!]
By Kreisel's basis theorem (see, e.g., \cite{Sa}, Theorem 7.2), it follows that $x\in L_{\omega_{1}^{\text{CK}}}$.
\end{proof}

\begin{lemma}
	For every $\alpha<\omega_{1}^{\text{CK}}$, there is an ITBM-recognizable real number $x$ such that $x\notin L_{\alpha}$. More specifically, if $\alpha<\omega_{1}^{\text{CK}}$ is an index, then $L_{\alpha+1}\setminus L_{\alpha}$ contains an ITBM-recognizable real number.
\end{lemma}
\begin{proof}
    The first claim follows from the second one, as there are cofinally many indices in $\omega_{1}^{\text{CK}}$. We thus show the second claim. Let $\alpha<\omega_{1}^{\text{CK}}$ be an index. If $\alpha<\omega^{\omega}$, every real number in $L_{\alpha+1}\setminus L_{\alpha}$ is ITBM-computable and thus ITBM-recognizable. So suppose that $\alpha\geq\omega^{\omega}$. 
    
    By Corollary \ref{many recognizable codes}, $L_{\alpha+1}$ contains an ITBM-recognizable code $c$ for $L_{\alpha}$. It thus suffices to see that $c\notin L_{\alpha}$. But it is clear that, as $c$ codes all real numbers contained in $L_{\alpha}$, we can define by diagonalization from $c$ a real number not contained in $L_{\alpha}$. For $c\in L_{\alpha}$, that real number would then be contained in $L_{\alpha}$ as well, a contradiction.
    
    %Let $\alpha<\omega_{1}^{\text{CK}}$. Without loss of generality, assume that $\alpha>\omega^{\omega}$. Let $\beta>\alpha$ be the next index $>\alpha$ and let $\gamma$ be the next index $>\beta$. By standard fine-structure, we have $\gamma<\omega_{1}^{\text{CK}}$. 
    %By Corollary \ref{many recognizable codes}, $L_{\gamma}$ has an ITBM-recognizable code $x$. 
    
    %Suppose for a contradiction that $x\in L_{\alpha}$. By choice of $\beta$, $L_{\gamma}$ contains a real number $r$ that is not contained in $L_{\alpha}$. But then, as $r$ is definable from $x$, we have $r\in L_{\alpha}[x]=L_{\alpha}$, a contradiction. 
    
    Thus $x$ is as desired.
    
	%Das Argument von oben zeigt, dass f\"ur jedes $\alpha<\omega_{1}^{\text{CK}}$ ein erkennbarer Code f\"ur $L_{\alpha}$ existiert. Die liegen vermutlich kofinal in $L_{\omega_{1}^{\text{CK}}}$ [CHECK THIS!].
\end{proof}

In combination with Lemma \ref{upper bound} above, this shows that new ITBM-recognizable real numbers appear wherever they can, i.e. at every $L$-level $<\omega_{1}^{\text{CK}}$ at which new real numbers appear at all. This is in contrast both with ITRMs and ITTMs, for which there are ``gaps'' in the set of constructible levels at which new recognizable ordinals appear, see, e.g. \cite{Ca1} or \cite{Ca}.

%Zur Verteilung, genauere Fragen: Liegen alle erkennbaren reellen Zahlen in $L_{\omega^{\omega}+1}}$? 
%Von denen sind jedenfalls diejenigen erkennbar, aus denen sich ein Code f\"ur $L_{\omega^{\omega}}$ berechnen l\"a{\ss}t.

%Vermutung: Nein, erkennbare k\"onnen deutlich h\"oher liegen. Aus der $\omega$-ten Iteration des Halteproblems auf einem Code f\"ur $L_{\omega^{\omega}}$ sollte man z.B. einen Code f\"ur $L_{\omega^{\omega}+2}$ kriegen, \"uber dem man dann weiter machen kann.

%Neue Frage: Liegen die erkennbaren alle in $L_{\omega^{\omega}+\omega}$? Nein. Das Argument von oben zeigt, dass f\"ur jedes $\alpha<\omega_{1}^{\text{CK}}$ ein erkennbarer Code f\"ur $L_{\alpha}$ existiert. Die liegen vermutlich kofinal in $L_{\omega_{1}^{\text{CK}}$ [CHECK THIS!].

%Andererseits: Da Berechnungen in der L\"ange durch $\omega^{\omega}$ beschr\"ankt sind, sind ITBM-erkennbare reelle zahlen immer $\Sigma_{1}^{1}$-singletons (snapshots sind durch reelle zahlen codierbar, eine $\omega^{\omega}$ lange Folge davon also ebenfalls, und die codierung ist einfach), nach Kreisels Basissatz liegen sie also in $L_{\omega_{1}^{\text{CK}}$.

\section{Non-recognizability}	

Given the results of the preceding section that the ITBM-recognizable real numbers appear cofinally in $L_{\omega_{1}^{\text{CK}}}$, it becomes natural to ask whether the same happens for ITBM-nonrecognizability. (Note that, for weak ITRMs, the set of recognizable real numbers coincides with $\mathbb{R}\cap L_{\omega_{1}^{\text{CK}}}$.) Moreover, since ITBMs increase in computational strength the more computational nodes are allowed in the program (so that, in particular, there is no universal ITBM) (see Koepke and Morozov, \cite{KM}), one may wonder whether the same happens for their recognizability strength (which is the case for ITRMs when one increases the number of registers, see \cite{Ca2}).
	
In this section, we will show that (i) non-ITBM-recognizable real numbers appear cofinally often in $L_{\omega_{1}^{\text{CK}}}$ and (ii) for every $n\in\omega$, there is a real number $x$ that is ITBM-recognizable (in fact ITBM-computable), but not ITBM-recognizable by a program with $\leq n$ many nodes.
	
The proof idea for both results is to consider real numbers that are Cohen-generic over sufficiently high $L$-levels below $L_{\omega_{1}^{\text{CK}}}$.\footnote{The same approach was used in \cite{Ca2} to obtain non-recognizables for ITRMs.} However, since we are working below the first admissible ordinal, the amount of set theory available in the relevant ground models is very small. Fortunately, forcing over the very weak set theory PROVI has been worked out by A. Mathias in Mathias \cite{Ma} and Bowler and Mathias \cite{BM}. The results from these papers that will be relevant below are summarized in the following lemma:

\begin{lemma}{\label{mathias stuff}} (Mathias and Bowler)

(i) [\cite{BM}] $L_{\alpha}$ is provident, i.e., a model of PROVI, if and only if $\alpha$ is indecomposable. In particular, $\omega^{\iota}$ is provident for all ordinals $\iota>0$. %steht im fließtext auf seite 32 von diesem dokument: https://www.google.com/url?sa=t&rct=j&q=&esrc=s&source=web&cd=&ved=2ahUKEwiD98Dmgv3rAhWCzKQKHUuNCmIQFjABegQIARAB&url=https%3A%2F%2Fpreprint.math.uni-hamburg.de%2Fpublic%2Fpapers%2Fhbm%2Fhbm613.pdf&usg=AOvVaw0Up-adsn9hCAMv2MvC6lHC (preprint-server der uni hamburg)

(ii) [\cite{Ma}, the forcing theorem for $\Delta_{0}$-formulas over provident sets] If $L_{\alpha}$ is provident, then the forcing theorem for $\Delta_{0}$-formulas (and forcings contained in $L_{\alpha}$) holds for $L_{\alpha}$. In particular, if $G$ is Cohen-generic over $L_{\alpha}$ and $\phi(\dot{a_{1}},...,\dot{a_{k}})$ is $\Delta_{0}$ (where $\dot{a_{1}},...,\dot{a_{n}}$ are names for Cohen forcing in $L_{\alpha}$) and $G$ is a Cohen-generic filter over $L_{\alpha}$ (i.e., $G$ intersects every dense subset of Cohen forcing that is contained in $L_{\alpha}$) then $L_{\alpha}[G]\models\phi(\dot{a_{1}}^{G},...,\dot{a_{k}}^{G})$ if and only if there is $p\in G$ such that $p\Vdash \phi(\dot{a_{1}},...,\dot{a_{k}})$.

\end{lemma}

\begin{lemma}{\label{generic nonrecog}}
	If $n\in\omega$ and $x\subseteq\omega$ is Cohen-generic over $L_{\omega^{n+1}}$, then $x$ is not ITBM-recognizable by an ITBM-program using $< n$ many nodes. In particular, if $x$ is Cohen-generic over $L_{\omega^{\omega}}$, then $x$ is not ITBM-recognizable.
\end{lemma}
\begin{proof}
	Suppose that $x$ is Cohen-generic over $L_{\omega^{n}}$ and ITBM-recognizable by the program $P$ which uses $< n$ nodes. By the bound in Koepke and Seyfferth \cite{KS} on ITBM-halting times, since $P^{x}$ halts, $P^{x}$ will run for less than $\omega^{n}$ many steps. Consequently, the halting computation of $P^{x}$ with output $1$ will be contained in $L_{\omega^{n}}[x]$. Thus $L_{\omega^{n+1}}[x]$ believes that $L_{\omega^{n}}[x]$ contains a halting computation of $P$ in the oracle $x$ with output $1$, which is a $\Delta_{0}$-formula. Let $\dot{A}$ be a name for $L_{\omega^{n}}[x]$ and let $\dot{x}$ be a name for $x$. By the forcing theorem for $\Delta_{0}$-formulas over provident sets, there is a condition $p\subseteq x$ that forces that $\dot{A}$ contains a halting computation of $P$ in the oracle $\dot{x}$ with output $1$. Now let $y$ be a real number that is Cohen-generic over $L_{\omega^{n+1}}$, extends $p$ and is different from $x$. Then $p$ will also force that $P^{y}$ halts with output $1$, contradicting the assumption that $P$ recognizes $x$. 
	
	If $x$ is Cohen-generic over $L_{\omega^{\omega}}$, it is in particular Cohen-generic over $L_{\omega^{n}}$ for every $n\in\omega$, thus not recognizable by an ITBM-program with any number of nodes, and thus not ITBM-recognizable. 
\end{proof}

\begin{thm}{\label{nonrecognizable distribution}}
	(i) For each $n\in\omega$, there is a real number $x$ that is ITBM-computable (and thus ITBM-recognizable), but not ITBM-recognizable by a program with $< n$ many nodes.
	
	(ii) For each $\alpha<\omega_{1}^{\text{CK}}$, there is an ITBM-nonrecognizable real number in $L_{\omega_{1}^{\text{CK}}}\setminus L_{\alpha}$. In fact, if $\alpha>\omega^{\omega}$ is an index, then $L_{\alpha+1}\setminus L_{\alpha}$ contains a non-ITBM-recognizable real number. Thus, below $\omega_{1}^{\text{CK}}$, new non-ITBM-recognizable real numbers appear wherever new real numbers appear at all.
	%In fact, for $\iota\geq\omega$, if $\omega^{\iota}$ is an index, then there is an ITBM-nonrecognizable real number in $L_{\omega^{\iota}+1}\setminus L_{\omega^{\iota}}$.
\end{thm}
\begin{proof}
	
	(i) For $n\in\omega$, let $x\in L_{\omega^{n+2}+1}$ be Cohen-generic over $L_{\omega^{n+1}}$. (The argument for the existence of such a real number is analogous to the one used in part (ii).) By Lemma \ref{mathias stuff}, $L_{\omega^{n+1}}$ is provident. However, as $x\in L_{\omega^{\omega}}$, it is ITBM-computable and thus ITBM-recognizable, but not by a program with $<n$ many nodes by Lemma \ref{generic nonrecog}.

	(ii) By Lemma \ref{generic nonrecog}, it suffices to show that $L_{\alpha+1}\setminus L_{\alpha}$ contains a Cohen-generic real number over $L_{\omega^{\omega}}$ whenever $\alpha\geq\omega^{\omega}$ is an index. 

	We will show that $L_{\alpha+1}$ in fact contains a real number that is Cohen-generic over $L_{\alpha}$, which will in most cases be much more than demanded. By fine-structure, a surjection from $\omega$ to $L_{\alpha}$ is definable over $L_{\alpha}$; a fortiori, there is a surjection $f$ from $\omega$ to the dense subsets of Cohen-forcing in $L_{\alpha}$ definable over $L_{\alpha}$, say by the formula $\phi_{f}(x,y,q)$, $q$ a finite tuple of elements of $L_{\alpha}$. Now define $x:\omega\rightarrow 2$ by letting $x(i)=b$ if and only if there is a finite sequence $(p_{j}:j\leq k)$ of Cohen-conditions such that $p_{0}=\emptyset$ and, for all $j<k$, $p_{j+1}$ is the $<_{L}$-minimal element of $f(j)$ that extends $p_{j}$ and such that $p_{k}(i)=b$. By definition of $x$, it is Cohen-generic and definable over $L_{\alpha}$.

	%ansatz: wird in der naechsten stufe abzaehlbar, das sollte fuer rasiowa-sikorski reichen. 
	%genauer: die rekursion benutzt immer nur endliche folgen, die als elemente in der stufe liegen, davon kann man jeweils die kleinste nehmen, die noch moeglich ist. 
   %noch genauer: die bijektion ist durch eine formel gegeben, wir k�nnen also sagen "das i-te bit ist b, falls eine endliche folge von bedingungen existiert, die einander erweitern, von denen die j-te ein element der j-ten dichten menge ist (dazu benutzen wir die formel, die die bijektion zwischen \omega und den dichten mengen definiert [surjektion reicht]) und bei der das letzte element an der i-ten stelle ein b hat".
\end{proof}

Note that the programs used in the proof of Corollary \ref{many recognizable codes} can all be taken to use the same number of nodes, so that there is a fixed number $n$ such that the real numbers that are ITBM-recognizable by programs with $n$ nodes are already cofinal in $L_{\omega_{1}^{\text{CK}}}$. 

This leaves us with the question whether, for any $n\in\omega$, there is a non-ITBM-computable real number that is ITBM-recognizable, but by a program with $n$ nodes. This is indeed the case.

\begin{thm}{\label{high bounded nonrecog}}
Let $n\in\omega$. Then there are cofinally in $L_{\omega_{1}^{\text{CK}}}$ many real numbers $x$ that are ITBM-recognizable, but not with $\leq n$ nodes. In particular, there are cofinally in $L_{\omega_{1}^{\text{CK}}}$ many ITBM-lost melodies that are not ITBM-recognizable with $\leq n$ nodes.
\end{thm}
\begin{proof}
	We will show that there is an ITBM-computable injection $f:\mathbb{R}\rightarrow\mathbb{R}$ with an ITBM-computable (partial) inverse function $g$ that maps each real number $x$ to a real number $\tilde{x}$ that is Cohen-generic over $L_{\omega^{n+1}}$. Once this is done, the result can be seen as follows: Let $P_{g}$ be an ITBM-program that computes $g$. Given $\alpha<\omega_{1}^{\text{CK}}$, pick a real number $c$ that is ITBM-recognizable, but not contained in $L_{\alpha}$ (the existence of such real numbers was proved above). Let $P$ be an ITBM-program for recognizing $c$. We claim that $\tilde{c}$ is ITBM-recognizable, but not with $\leq n$ nodes. The latter claim follows since $\tilde{c}$ is by definition Cohen-generic over $L_{\omega^{n+1}}$. To recognize $\tilde{c}$, first check whether $g(\tilde{c})$ is defined. Note that this can by clocking $P_{g}^{\tilde{c}}$ for $\omega^{m+1}$ many steps, where $m$ is the number of nodes used in $P_{g}$. If not, halt with output $0$. Otherwise, compute $g(\tilde{c})$ and return the output of $P^{g(\tilde{c})}$. Since $g$ is injective, this will give the right result.
	
	The encoding works as follows: Let $\mathcal{D}=(D_{i}:i\in\omega)$ be an ITBM-computable enumeration of the dense subsets of Cohen-forcing contained in $L_{\omega^{n+1}}$, encoded in some natural way as a real number $d$. Define a sequence $(\tilde{c}_{i}:i\in\omega)$ by letting $c_{0}=\emptyset$ and 
	$\tilde{c}_{i+1}$ the lexically minimal element $e_{i}$ of $D_{i}$ that extends $\tilde{c}_{i}$ when $i\in c$ and otherwise $\tilde{c}_{i+1}$ is the lexically first element of $D_{i}$ that properly extends $\tilde{c}_{i}$ and is incompatible with $e_{i}$. Then let $\tilde{c}:=\bigcup_{i\in\omega}\tilde{c}_{i}$. It is now easy to see that there the function $h:\mathbb{R}\times \omega\rightarrow\mathbb{R}$ that maps $(x,i)$ to the $i$th bit of $\tilde{x}$ is actually recursive in $d$.
	
	Similarly, we can recursively reconstruct $x$ from $\tilde{x}$ and $d$: Namely, given $x$, $i$ and $d$, compute a sufficiently long initial segment of $(\tilde{x}_{i}:i\in\omega)$ such that the last element fixes the $i$th bit. To compute $\tilde{x}_{i+1}$ from $\tilde{x}_{i}$, exhaustively search (in lexicographic order) through all finite partial functions from $\omega$ to $2$ that extend $\tilde{x}_{i}$, find the lexically minimal element $e_{i}$ that properly extends $\tilde{x}_{i}$ and see whether $x$ extends $e_{i}$. If that is the case, then $i\in x$, otherwise, we have $i\notin x$. 
	
	The second claim now follows, as the ITBM-computable real numbers belong to $L_{\omega^{\omega}}$. 
\end{proof}

%\section{Conclusion and Further Work}
%This leaves us with the following question, which remains unanswered for the time being:

%\begin{question}
%	Is it true that, for any $n\in\omega$, there is a non-ITBM-computable real number that ITBM-recognizable, but by a program with $n$ nodes?
%\end{question}


\begin{thebibliography}{widestlabel}
\bibitem[BM]{BM} N. Bowler, A. Mathias. Rudimentary recursion, gentle functions and provident sets. Notre Dame Journal of Formal Logic 56(1) (2015)
\bibitem[BP]{BP} G. Boolos, H. Putnam. Degrees of Unsolvabllity of constructible sets of integers, J. Symbolic Logic 33 (1968) 
\bibitem[Ca]{Ca} M. Carl. Ordinal Computability. An introduction to infinitary machines. De Gruyter Berlin/Boston (2019)
\bibitem[Ca1]{Ca1} The distribution of ITRM-recognizable reals. Ann. Pure Appl. Log. 165(9): 1403-1417 (2014)
\bibitem[Ca2]{Ca2} Optimal Results on ITRM-recognizability.  J. Symb. Log. 80(4): 1116-1130 (2015)
\bibitem[Ca3]{Ca3} The lost melody phenomenon. In: S. Geschke et al. (eds.): Infinity, Computability, and Metamathematics : Festschrift celebrating the 60th birthdays of Peter Koepke and Philip Welch.  London : College Publ.(Tributes ; 23) (2014)
\bibitem[CFKMNW]{CFKMNW} M. Carl, T. Fischbach, P. Koepke, R. Miller, M. Nasfi, G. Weckbecker. An enhanced theory of infinite time register machines. In: A. Beckmann et al. (eds.): Logic and Theory of Algorithms.  Lecture Notes in Computer Science 5028 (2008),
\bibitem[CSW]{CSW} M. Carl, P. Schlicht, P. Welch. Recognizable sets and Woodin cardinals.	Ann. Pure Appl. Log. 169(4): 312-332 (2018)
	
\bibitem[HL]{HL} J. Hamkins, A. Lewis. Infinite Time Turing Machines. J. Symbolic Logic. 65(2) (2000)

\bibitem[Je]{Je} R. Jensen. The fine structure of the constructible hierarchy. Ann. of Math. Logic 4(3) (1972)


\bibitem[Ko1]{Ko1} P. Koepke. Infinite Time Register Machines. In: A. Beckmann et al. (eds.): Logical Approaches to Computational Barriers. A. Beckmann et al. (eds.) Lecture Notes in Computer Science 3988 (2006),
\bibitem[KM]{KM} P. Koepke, R. Miller. An enhanced theory of infinite time register machines. In: A. Beckmann et al. (eds.):  Logic and Theory of Algorithms. Lecture Notes in Computer Science 5028 (2008)

\bibitem[KS]{KS} P. Koepke, B, Seyfferth. Towards a theory of Infinite Time Blum-Shub-Smale Machines.  In: S. Cooper et al. (eds.): How the World Computes. CiE 2012. Lecture Notes in Computer Science, vol 7318. Springer Berlin (2012)

\bibitem[KM]{KM} P. Koepke, A. Morozov. The computational strength of infinite time blum-shub-smale machines. Algebra and Logic 56(1) (2017)

\bibitem[Ma]{Ma} A. Mathias. Provident sets and rudimentary set forcing. Fundamenta Mathematicae 230 (2015)

\bibitem[Sa]{Sa} G. Sacks. Higher Recursion Theory. Perspectives in Mathematical Logic, Volume 2. Springer Berlin (1990)
\bibitem[We]{We} P. Welch. Discrete Transfinite Computations. In: G. Sommaruga, T. Strahm (eds.): Turing's Revolution. The Impact of His Ideas about Computability. Birkh\"auser Basel (2015)
\end{thebibliography}
\end{document}